\RequirePackage{fix-cm}
\documentclass[smallextended]{svjour3}       
\smartqed  
\usepackage{graphicx}
\pdfoutput=1
\usepackage{tabu}
\usepackage{amssymb}
\usepackage{float}
\usepackage{url}
\usepackage{algorithmic}
\usepackage{algorithm}
\usepackage{color}
\usepackage{soul}
\usepackage{breqn}

\newtheorem{cor}{Corollary}
\begin{document}

\title{On a borderline between the NP-hard and polynomial-time solvable cases of the flow shop with job-dependent storage requirements
}

\titlerunning{On a borderline between the NP-hard and polynomial-time solvable cases}        

\author{Alexander Kononov        \and
        Julia Memar                 \and
        Yakov Zinder
}

\authorrunning{A. Kononov et al.}

\institute{A. Kononov \at
              Sobolev Institute of Mathematics, Siberian Branch of the Russian Academy of Sciences, Novosibirsk, Russia\\
              \email{alvenko@math.nsc.ru}           
           \and
           J. Memar \at
            University of Technology Sydney, PO Box 123, Broadway, NSW, 2007, Australia\\
            \email{julia.memar@uts.edu.au}
            \and
            Y. Zinder \at
            University of Technology Sydney, PO Box 123, Broadway, NSW, 2007, Australia\\
            \email{yakov.zinder@uts.edu.au}
}
\date{Received: date / Accepted: date}

\maketitle

\begin{abstract}
The paper is concerned with the two-machine flow shop, where each job requires an additional resource (referred to as storage space) from the start of its first operation till the end of its second operation. The storage requirement of a job is determined by the processing time of its first operation. At any point in time, the total consumption of this additional resource cannot exceed a given limit (referred to as the storage capacity). The goal is to minimise the makespan, i.e. to minimise the time needed for the completion of all jobs. This problem is NP-hard in the strong sense. The paper analyses how the parameter - a lower bound on the storage capacity specified in terms of the processing times, affects the computational complexity.


\end{abstract}

\section{Introduction}\label{sec:intro}

This paper considers the two-machine flow shop that is comprised of the first-stage machine $M_1$, the second-stage machine $M_2$, and an additional resource referred to as a buffer or storage space. The machines must process a set of jobs $N = \{1,...,n\}$.  In order to be processed a job $j \in N$ must be processed on the machine $M_1$ for $a_j$ units of time and, after that, on the machine $M_2$ for $b_j$ units of time. Once a machine starts to process a job, it continues processing till the completion of this operation. Each machine can process at most one job at a time. For each job, the processing on $M_2$ can commence only after the completion the processing of this job on $M_1$. Each job $j$ seizes $\omega(j)$ units of storage space (buffer) at the beginning of its processing on $M_1$ and releases this portion of storage space only at the completion of the job's processing on $M_2$. Similar to \cite{kononov2012quantity,kononov2019flow,kononova2013variable,lin2013sequence,lin2008minimize} it is assumed that, for each $j \in N$,  $\omega(j)=a_j$. At any point in time, the total consumption of the storage space cannot exceed $\Omega$ - the storage capacity.

The processing of jobs commences at time $t = 0$. A schedule $\sigma$ specifies for each $j \in N$ and each machine $M_i$, the point in time $S_j^i(\sigma)$ when this job commences its processing on $M_i$. Since each operation is processed without interruptions, for any $j\in N$,  the completion times of job $j$ on $M_1$ and $M_2$ are
\[
  C_j^1(\sigma) = S_j^1(\sigma) + a_j \hspace{0.3cm}\mbox{and}\hspace{0.3cm} C_j^2(\sigma) = S_j^2(\sigma) + b_j,
\]
respectively. The goal is to minimise the makespan
\[
 C_{max}(\sigma)=\max_{j\in N}C_j^2(\sigma).
\]
Following \cite{kononov2012quantity}, the considered problem will be referred to as PP-problem.

The flow shop problems with a buffer have been extensively studied in the literature on scheduling, but most of these publications consider flow shops with an intermediate buffer between stages and assume that this buffer limits only the number of jobs that have completed one operation and are waiting for the commencement of the next one \cite{brucker2003flow,brucker2012resource,emmons2013hybrid,papadimitriou1980flowshop,pinedo2016scheduling}. In contrast, in the scheduling problem, considered in this paper, the storage (buffer) requirement varies from job to job and each job seizes the required storage space during its processing on the machines as well as during its waiting time between the operations. This also differentiates this paper from the publications on resource constrained scheduling that assumes that the additional resource is consumed by a job only during its processing on the machines but it is released between the operations \cite{blazewicz1983scheduling,brucker2012resource}.

Although the practical significance of the flow shops with job-dependent storage requirements where a job uses this resource not only during its processing on the machines but also between the operations was acknowledged as early as in \cite{witt2007simple}, the active research in this field of scheduling theory was triggered by the publications \cite{lin2008minimize} and \cite{lin2009two} which were concerned with controlling the lag during a multimedia presentation of media objects, where each object has a loading time and a play time. An object's presentation cannot be started until it has finished loading. To avoid the potential presentation delay that occurs while waiting for loading, the player is designed to ``prefetch'' the objects prior to their intended presentation time. For applications that use portable devices such as PDAs to play multimedia objects, one must consider the memory size allowed for storing the objects. Under a network environment with a rather stable transmission rate, the download time of an object is proportional to its file size. Since, we can scale the file size in such a way that a unit of size is loaded per a unit of time. So, without loss of generality, we can assume that  the download time of an object is equal to the size of its file. This leads to the model where the duration of the operation on the first-stage machine is equal to the storage space requirement.

The two-machine flow shop with job-dependent storage requirements, considered in this paper, was used in the study on data gathering in networks with star topology \cite{berlinska2015scheduling,berlinska2020heuristics}. Such a network is comprised of a number of worker nodes, each transmitting its data set to the base station where this data set is to be processed. The base station can communicate with at most one worker node at a time and allocates to a data set the entire required memory from the beginning of its transmission till the end of its processing by the station.

The practical applications of the considered flow shop with job-dependent storage requirements can be found far beyond the confines of the systems where storage is the system computer memory. Thus, the interest of one of the coauthors of this paper in the PP-problem was originally inspired by the scheduling situations in supply chains where the change of the means of transportation involves unloading and loading, using certain storage space. Such situations arise, for example, in a supply chain of mineral resources  \cite{fung2015capacity}.

One of the key postulates of the PP-problem is the assumption that the storage space, required by a job, is determined by the processing time of its first operation. As it has been discussed above, this assumption is justified by several applications. On the other hand, it is known that, in manufacturing, the condition that the storage requirement is determined by the duration of one of the operations may be inadequate \cite{witt2007simple}. The generalisation of the PP-problem where the storage requirement of a job is not necessarily determined by the processing time of one of its operations was considered in \cite{zinder20215}. Other variations and generalisations of the PP-problem, considered in the literature, include the relaxation of the assumption that the storage capacity is a constant  \cite{berlinska2020two}; the replacement of the assumption that there exists only one pair of machines by the assumption that each job can be processed by one of several disjoint pairs of machines, each pair being assigned a storage capacity, which varies from pair to pair \cite{ernst2018flexible}; a flexible flow shop with batch processing \cite{gu2019improved}; as well as different objective functions.

The PP-problem is NP-hard in the strong sense \cite{lin2009two}. The problem remains NP-hard in the strong sense even under the restriction that, on one of the machines, the jobs are to be processed in a given sequence  \cite{GU2018143}. Furthermore, it has been proven in \cite{fung2016permutation} that there are instances for which the set of all optimal schedules does not contain a permutation schedule, that is, a schedule in which both machines process the jobs in the same order. Even the decision problem, requiring an answer to the question of whether or not the given instance is one of the instances that do not have an optimal schedule that is a permutation one, is NP-complete.

The computational complexity of the PP-problem motivated interest in branch-and-bound algorithms  \cite{kononov2012quantity,lin2009two,lin2013sequence} as well as in integer programming-based and metaheuristic optimisation methods \cite{kononova2013variable,le2020iterated}. Another direction of research, triggered by the computational complexity of the PP-probem, was the study of its various particular cases \cite{berlinska2020heuristics,kononov2019flow,le2019iterated,min2019two,zinder20215}. Thus, \cite{kononov2019flow} presents a polynomial-time algorithm and a proof that this algorithm constructs an optimal schedule for any instance of the PP-problem where the storage capacity is not less than five times the maximal processing time. This result raised the question of how the computational complexity changes with the variation of a lower bound on the storage capacity computed as a function of the processing times.

The remainder of this paper addresses this question. Section \ref{sec:NPH} presents a proof that the PP-problem remains NP-hard after the introduction of the additional assumption that, for any $\delta >0$,
\begin{equation}\label{Buffer/4}
     \frac{4}{1+\delta } \max_{i\in N} \{a_i, b_i \} \leq \Omega.
\end{equation}
Section \ref{sec:PolynomialCase} presents a polynomial-time algorithm and a proof that this algorithm constructs an optimal schedule if
\[
     3.5\max_{i\in N} a_i +\max \left \{0.5\max_{i\in N} a_i, \max_{i\in N} b_i \right \} \leq \Omega.
\]
Observe that these results not only contribute to the efforts aimed at establishing the borderline between the NP-hard and polynomial-time solvable, but also significantly strengthen the result in \cite{kononov2019flow} by replacing the factor 5 in the lower bound in \cite{kononov2019flow} by 4.5.

\section{NP-hardness}\label{sec:NPH}

We show that the PP-problem remains NP-hard even if (\ref{Buffer/4}) holds. We use the following NP-hard variant of the partition problem.
The input consists of $2m$ positive integers $e_1, e_2, \ldots, e_{2m}$ such that
\[\sum_{k=1}^{2m} e_k=2E \mbox{   and    }\frac{E}{m+1} < e_k < \frac{E}{m-1}, \;\;\;\;k=1,\ldots,2m.\]
It is necessary to answer the following question: does there exist an index set $K,$ such that $\sum_{k \in K} e_k= E$ holds?  Without lost of generality we assume that
\begin{equation}\label{delta}
    m \geq \max\{2(1+\frac{1}{\delta}), 8\}.
\end{equation}

Consider the following instance of the PP-problem with $2m+4$ jobs. We have $2m$ regular jobs. For $j=1,\ldots,2m$ a regular job $j$ has processing times
$a_j=E$ and $b_j=E+e_j$, and $a_0=0$ and $b_0=E$ for the job $0$.  There are three special jobs
$2m+1,$ $2m+2,$ and $2m+3.$ For each special job $j$, let $a_j=E$ and $b_j=0$. The buffer capacity $\Omega = 4E(1-\frac{1}{(m-1)}) < 4E.$ We are going to check whether or not there exists a feasible schedule with makespan $C_{max} \leq (2m+3)E.$

It is easy to see that no four jobs fit in the buffer unless the job $0$ is one of the jobs. On the other hand, any three jobs can be in the buffer at the same time. Indeed, taking into account (\ref{delta}), we have
$$a_i+a_j+a_k < 3E + \frac{3E}{m-1} = 3E + \frac{7E}{m-1} - \frac{4E}{m-1}  \leq 4E \left (1-\frac{1}{(m-1)} \right ) = \Omega.$$
Furthermore, all jobs $0\leq i\leq 2m+3$, satisfy the condition (\ref{Buffer/4}):
$$\max_{i\in N} \{a_i, b_i \} \leq E + \frac{E}{m-1} = E\left( \frac{m}{m-1}\times\frac{m-2}{m-2} \right)= E \left( \frac{m-2}{m-1}\times\frac{m}{m-2}\right)=$$
$$=E \left( 1- \frac{1}{m-1} \right ) \left (1+\frac{2}{m-2} \right )\leq \frac{\Omega}{4}\left (1+\delta \right ).$$

Let $\sigma$ be a schedule with makespan $C_{max} \leq (2m+3)E.$ We note that the load on both machines is equal to $(2m+3)E.$
Thus, both machines are not idle during the interval $[0, (2m+3)E].$ It follows that $0$ is the first job in $\sigma$, and
the job sequences on the both machines are the same. Observe that the machine $M_2$ works without idle time if and only if for each job $i$ its completion
time $C^1_{i}(\sigma)$ does not exceed the completion time of the job which precedes it on $M_2.$ We note that $a_j=E$ for all jobs except the job $0$, and it follows that for all jobs in $\sigma$ except for the last job, $C^2_{i}(\sigma) - C^1_{i}(\sigma) \geq E.$

Without loss of generality we assume that the job $2m+1$ is the first special job in $\sigma$. Assume that job $i>0$ immediately precedes the job $2m+1.$

If $C^2_{i}(\sigma) - C^1_{i}(\sigma) < 2E,$ then $C^2_{2m+1}(\sigma) - C^1_{2m+1}(\sigma) < E$, which implies that the machine $M_2$ is idle
and hence $C_{max}(\sigma) > (2m+3)E.$  Suppose that $C^2_{i}(\sigma) - C^1_{i}(\sigma) > 2E.$ Since the machine $M_1$ processes the jobs without idle time,
then there are at least four jobs $j$ with $a_j=E$ in the buffer during the interval $[C^1_{i}(\sigma)+2E, C^2_{i}(\sigma)],$ which violates the buffer capacity. Thus,  $C^2_{i}(\sigma) - C^1_{i}(\sigma) = 2E.$

Denote by $K'$ the set of jobs that precede the job $2m+1$ in $\sigma.$ We have
\[2E = \sum_{i\in K'}b_i -  \sum_{i\in K'}a_i = b_0 + \sum_{i\in K} e_i = E + \sum_{i\in K } e_i.\]

Therefore, a feasible schedule with makespan $C_{max} \leq (2m+3)E$ exists if and only
if for $2m$ positive integers $e_1, e_2, \ldots, e_{2m}$ such that
\[\sum_{k=1}^{2m} e_k=2E \mbox{   and    }\frac{E}{m+1} < e_k < \frac{E}{m-1},  \;\;\;\;k=1,\ldots,2m,\]
there exists an index set $K$ such that $\displaystyle \sum_{i\in K } e_i = E$.

\section{Polynomial-time algorithm}\label{sec:PolynomialCase}
The PP-problem is strongly NP-hard \cite{lin2008minimize}. However, the computational complexity of the problem depends on the relationship between the size of jobs and the size of the buffer.
Indeed, the problem is easily solvable if the buffer size is large enough, for example, when all jobs can be simultaneously placed in the buffer.
In this case, the problem is equivalent to the two-machine flow shop problem without buffer and it can be solved in $O(n \log n)$ time by Johnson's algorithm \cite{johnson1954optimal}.
Johnson's rule can be stated as follows:
\begin{itemize}
  \item partition $N$ into two sets:
  \[
   L_1=\{i\in N: a_i < b_i\}
 \hspace{0.5cm} \mbox{and} \hspace{0.5cm}
   L_2=\{i\in N: a_i \geq b_i \};
  \]
  \item first schedule the jobs from $L_1$ in non-decreasing order of $a_i,$ and then schedule the jobs from $L_2$ in non-increasing order of $b_i$.
\end{itemize}

Denote by $\sigma^J$ the schedule constructed by Johnson's rule. The makespan $C_{max}(\sigma^J)$ is a lower bound  on the objective value of the PP-problem. Let
\[
 a_{max}=\max_{i\in N} a_i
 \hspace{0.5cm} \mbox{and} \hspace{0.5cm}
 b_{max}=\max_{i\in N} b_i.
\]
The polynomial-time algorithm below  constructs a schedule for an extended set of jobs obtained by adding some auxiliary jobs. The set of these auxiliary jobs is  determined by  schedule $\sigma^J$.  It is shown that if
\begin{equation}\label{Buffer/5}
    3.5a_{max} +\max\{0.5a_{max}, b_{max}\} \leq \Omega,
\end{equation}
then the makespan of this schedule is equal to $C_{max}(\sigma^J)$ which implies its optimality.

Number the jobs according to the sequence constructed by Johnson's rule, then
\begin{equation}\label{Johnson'sCmax}
  C_{max}(\sigma^J) = \max_k \left(\sum_{i=1}^k a_i + \sum_{i=k}^n b_i\right).
\end{equation}
Denote by $Idle_1$ and $Idle_2$ the total idle time in the interval $[0,C_{max}(\sigma^J)]$ on machines $M_1$ and $M_2,$ respectively. If the maximum in (\ref{Johnson'sCmax}) is obtained for $k = k'$, then
$Idle_1= C_{max}(\sigma^J) - \sum_{i=1}^n a_i  \leq \sum_{i=k'}^n b_i $ and $Idle_2= C_{max}(\sigma^J) - \sum_{i=1}^n b_i   \leq \sum_{i=1}^{k'} a_i.$

Let $X$ and $Y$ be two sets of auxiliary jobs of cardinality
\[
 |X| =  \left \lceil \frac{Idle_2}{b_{max}} \right \rceil
 \hspace{0.3cm} \mbox{and} \hspace{0.3cm}
 |Y| =  \left \lceil \frac{Idle_1}{a_{max}} \right \rceil
\]
and such that
\[
 a_i = \left\{\begin{array}{cl}
 			0, & \mbox{if } i \in X \\[3pt]
		\displaystyle	 \frac{Idle_1}{|Y|} , & \mbox{if } i \in Y
 		\end{array}
	\right.
 \hspace{0.5cm} \mbox{and} \hspace{0.5cm}
 b_i = \left\{\begin{array}{cl}
 		\displaystyle	\frac{Idle_2}{|X|}, & \mbox{if } i \in X \\[5pt]
			0, & \mbox{if } i \in Y
 		\end{array}
	\right. .
\]
Observe that
\begin{equation}\label{sum_a=sum_b}
  \sum_{i\in N \cup X \cup Y}a_i=\sum_{i\in N \cup X \cup Y}b_i,
\end{equation}
and that $a_i\leq a_{max}$ and $b_i\leq b_{max}$ for any job $i\in N \cup X \cup Y$. Consider a permutation schedule (a permutation schedule is a schedule with the same order of jobs on both machines) $\sigma'$ where the jobs of the set $N' = N \cup X \cup Y$ are scheduled as follows. The first $|X|$ jobs are the jobs constituting $X$, sequenced in arbitrary order; these jobs are followed by all jobs in $N$, scheduled according to Johnson's rule; the jobs from $N$ are followed by the arbitrary ordered remaining jobs, i.e. the jobs constituting $Y$. It is easy to see that
\[
 C_{max}(\sigma') = C_{max}(\sigma^J)
\]
and that $\sigma'$ can be viewed as a result of the application of Johnson rule to $N'$.

Let $\pi^J$ be the permutation of all jobs in $N'$ induced by the order in which these jobs are processed in $\sigma'$ and let $\pi_0$, $\pi_1$ and $\pi_2$ be the permutations of the jobs in
\[
	\begin{array}{c}
 L_0'=X\cup \left \{i \in L_1| \;\; a_i \leq \displaystyle \frac{1}{2}a_{max} \right \}, \\
\rule{0pt}{4ex}
 L_1' = \left\{i \in L_1| \;\; a_i >  \displaystyle  \frac{1}{2}a_{max}\right \}, \\
\rule{0pt}{4ex}
 L_2'=L_2\cup Y,
 	\end{array}
\]
respectively, induced by $\pi^J$.

The algorithm below constructs a permutation $\pi$ that specifies the order in which the jobs are processed in an optimal permutation schedule. It is convenient to use the following notation:
\[
 \begin{array}{ll}
 	\Delta_j = b_j - a_j, & \mbox{ for all } j \in N'; \\
\rule{0pt}{4ex}
	 \mu(k)=\displaystyle \sum_{j=k}^{|L_0'|} \Delta_{\pi_0(j)} , & \mbox{ for all } 1 \le k \le |L_0'|; \\
\rule{0pt}{3ex}
	 n' = n + |X| + |Y| ; &	 \\
\rule{0pt}{3ex}
	\displaystyle   l_{k,1}(\pi) = \sum_{j=1}^{k} a_{\pi(j)},  & \mbox{ for all } 1 \le k \le n'; \\
\rule{0pt}{3ex}
	\displaystyle  l_{k,2}(\pi) = \sum_{j=1}^{k} b_{\pi(j)}, & \mbox{ for all } 1 \le k \le n'; \\
\rule{0pt}{3ex}
	R_k(\pi) = l_{k,2}(\pi) - l_{k,1}(\pi), & \mbox{ for all } 1 \le k \le n'.
 \end{array}
\]
Let $\pi = \emptyset$ indicate that the permutation $\pi$ is not specified.

\begin{algorithm}[H]
\begin{algorithmic}[1]
\STATE Set $i=1, i_0=1, i_1=1, i_2=1, \pi=\emptyset, R_0(\pi) = 0, l_{0,1} = 0, l_{0,2} = 0$.
\WHILE {$ i_0 \leq |L_0'|$ }
           \IF {($R_{i-1} (\pi) < 2a_{max}$  {\bf and} $R_{i-1}(\pi)  + \Delta_{\pi_2(i_2)} + \mu(i_0) < a_{max}$) {\bf or} $R_{i-1} (\pi) < \frac{3}{2}a_{max}$}
                 \STATE set $\pi(i)=\pi_0(i_0)$, $i_0=i_0+1$;
                 \ELSE
                 \STATE set $\pi(i)=\pi_2(i_2)$, $i_2=i_2+1$;
           \ENDIF
            \STATE set $S_{\pi(i)}^1(\sigma)=l_{i-1,1}(\pi)$ and $S_{\pi(i)}^2(\sigma)=l_{i-1,2}(\pi)$; set  $i=i+1;$
\ENDWHILE
\WHILE {$i\leq n'$ }
           \IF {$R_{i-1}(\pi) < 2a_{max}$ {\bf and}  $i_1 \leq |L_1'|$}
                \STATE set $\pi(i)=\pi_1(i_1)$, $i_1=i_1+1$;
                 \ELSE
                \STATE set $\pi(i)=\pi_2(i_2)$, $i_2=i_2+1$;
           \ENDIF	
               \STATE set $S_{\pi(i)}^1(\sigma)=l_{i-1,1}(\pi)$ and $S_{\pi(i)}^2(\sigma)=l_{i-1,2}(\pi)$, set  $i=i+1$;
\ENDWHILE
\RETURN schedule $\sigma.$
\end{algorithmic}
\caption{}\label{algorithm}
\end{algorithm}

\begin{lemma}\label{while 1}
 Algorithm 1 schedules all jobs in $N'$.
\end{lemma}
\begin{proof}
Taking into account the condition in line 2, the first while loop (line 2 - 9) schedules all jobs in $L_0'$ (and probably some jobs in $L_2'$) if, each time when line 4 is to be executed, there exists an unscheduled job from $L_0'$ and,
each time when line 6 is to be executed, there exists an unscheduled job from $L_2'$. Observe that the condition in line 2 guarantees this for line 4.

As far as the remaining jobs are concerned, by virtue of the condition in line 10, they all will be scheduled by the second while loop (lines 10 - 17) if, each time when line 12 is to be executed, there exists an unscheduled job from $L_1'$ and, each time when line 14 is to be executed, there exists an unscheduled job from $L_2'$. Observe that
the condition in line 11 guarantees this for line 12 whereas the condition in line 10 together with the equality
\begin{equation}\label{i1=|L1|+1}
 i_1 = |L_1'| + 1
\end{equation}
guarantees this for line 14.

If the execution of line 14 is needed and (\ref{i1=|L1|+1}) does not hold, then, for the corresponding $i$,  $R_{i-1}(\pi) > 0$. Such inequality also holds each time when the execution of line 6 is needed. Since, for any $1 \le i \le n'$,
\begin{eqnarray}
&&R_{i-1}(\pi) < R_{i}(\pi), \mbox{  if  }\pi(i) \in L_0' \cup L_1' \label{eqn:R_up}\\
&&R_{i-1}(\pi)  \geq R_{i}(\pi), \mbox{  if  }\pi(i) \in L_2', \label{eqn:R_down}
\end{eqnarray}
and since (\ref{sum_a=sum_b}) implies that $R_{n'}(\pi) = 0$, the inequality $R_{i-1}(\pi) > 0$ implies the existence of an unscheduled job from $L_2'$. \qed
\end{proof}

\begin{theorem}\label{lem:alg1buidlds opt}
 If (\ref{Buffer/5}) holds, then Algorithm 1 constructs an optimal schedule.
\end{theorem}
\begin{proof}
The schedule $\sigma$, constructed by Algorithm 1, is feasible if
\begin{enumerate}
\item[\textbf{(a)}] the schedule $\sigma$ has no overlapping jobs on the same machine;
\item[\textbf{(b)}] the schedule $\sigma$ has no overlapping operations of the same job;
\item[\textbf{(c)}] the schedule $\sigma$ does not violate the buffer constraint.
\end{enumerate}

\noindent \textbf{(a)}: This condition holds because, for each $j \geq 2$,
\[
 S_{\pi(j)}^1(\sigma)=l_{j-1,1}(\pi) = l_{j-2,1}(\pi) + a_{\pi(j-1)} = S_{\pi(j-1)}^1(\sigma) + a_{\pi(j-1)} = C_{\pi(j-1)}^1(\sigma),
\]
\[
 S_{\pi(j)}^2(\sigma)=l_{j-1,2}(\pi) = l_{j-2,2}(\pi) + b_{\pi(j-1)} = S_{\pi(j-1)}^2(\sigma) + b_{\pi(j-1)} = C_{\pi(j-1)}^2(\sigma).
\] \\
\noindent\textbf{(b)}:  In what follows, we will consider a few cases to show that there are no operations overlapping in schedule $\sigma$. According to the algorithm any job from $L_0'$ precedes any job from $L_1'$ in the permutation $\pi.$ Assume that
\begin{equation}\label{condition}
    \pi(|L_0'|+|L_1'|+1)=\pi_2(1).
\end{equation}
The equality above implies that all jobs from $L_2'$ succeed any job from $L_1'$ in the permutation $\pi.$ Then it follows that the permutation $\pi$ is Johnson's permutation and the schedule $\sigma$ coincides with the schedule $\sigma'$. The feasibility $\sigma'$ implies that the schedule $\sigma$ has no overlapping operations of the same job.

If (\ref{condition}) does not hold, then in the permutation $\pi$ the job $\pi_2(1)$ is before the job $\pi_1(|L_1'|),$ i.e. $\pi(i)=\pi_2(1)$ for some $i < |L_0'|+|L_1'|+1.$
Thus  either there exists $k \leq |L_0'|$ such that
\[R_{k-1}(\pi)  \geq \frac{3}{2}a_{max}\mbox{ and } R_{k-1}(\pi)  + \Delta_{\pi_2(1)} + \mu(k) \geq a_{max},\]
or there exists $k$ , $ k \leq |L_0'|+|L_1'|$ such that $R_{k-1}(\pi)  > 2a_{max}.$

Let $\pi(h_0)$ be the last job from $L_0'$ and $\pi(h_1)$ be the last job from $L_1'$ in the permutation $\pi,$ i.e., $\pi(h_0)=\pi_0(|L_0'|)$ and $\pi(h_1)=\pi_1(|L_1'|).$ Note that $k < h_1,$ as otherwise we would have (\ref{condition}) satisfied. Partition the permutation $\pi$ into three subsequences: $(\pi(1),\ldots,\pi(k-1)),$ $(\pi(k),\ldots,\pi(h_1)),$ and $(\pi(h_1+1),\ldots,\pi(n')).$

\textbf{Case 1(b): $1 \leq j \leq k-1 $. }
The first $k-1$ jobs in the $\pi$ are from the set $L_0' \cup L_1'$, hence $\sigma$ and $\sigma'$  are the same for the first $k-1$ jobs.
Consequently, for $j=1, \ldots, k-1$, the operations of job $\pi(j)$ do not overlap.

\textbf{Case 2(b): $k \leq j \leq h_0.$}
Observe that this case is possible only if $k \leq |L_0'|$.
Consider two sub-cases, which are defined by the set the job $\pi(j)$ belongs to: either $\pi(j)\in L_2'$ or $\pi(j)\in L_0'$.

\textbf{Case 2.1(b): $\pi(j) \in L_2'$.} According to the algorithm either  $R_{j-1}(\pi)  > 2a_{max}$ or $R_{j-1}(\pi)  \geq \frac{3}{2}a_{max}$ and $R_{j-1}(\pi)  + \Delta_{\pi(j)} + \mu(i_0) \geq a_{max},$
where $i_0$ is specified according to Algorithm 1.
Thus, we have
\begin{equation} \label{eqnNB}
R_{j-1}(\pi)  \geq \min\{ 2a_{max}, a_{max} - \Delta_{\pi(j)} - \mu(i_0)\}.
\end{equation}
We have
\begin{eqnarray}
    C_{\pi(j)}^1(\sigma) && = S_{\pi(j)}^1(\sigma) + a_{\pi(j)} = l_{j-1,1}(\pi) + a_{\pi(j)} = l_{j-1,2}(\pi) + a_{\pi(j)} - R_{j-1}(\pi) \nonumber\\
    &&\leq l_{j-1,2}(\pi) + a_{\pi(j)} -  \frac{3}{2}a_{max} \leq  l_{j-1,2}(\pi) = S_{\pi(j)}^2(\sigma).\nonumber
\end{eqnarray}
Thus, the operations of the job $\pi(j)$ do not overlap. Moreover, for the $\pi(j) \in L_2'$ the following inequality holds:
 $$
 R_{j}(\pi) = R_{j-1}(\pi) +b_j - a_j \geq \frac{3}{2}a_{max} - a_{max} \geq \frac{1}{2}a_{max}.\label{eqn:lb1}
$$
\noindent Taking into account (\ref{eqn:R_up}), we have that $R_{j}(\pi) \geq \frac{1}{2}a_{max}$ for all $k+1 \leq j \leq h_0.$

\textbf{Case 2.2(b):  $\pi(j) \in L_0'$. }
We have
\begin{eqnarray}
   C_{\pi(j)}^1(\sigma)&& = S_{\pi(j)}^1(\sigma) + a_{\pi(j)} = l_{j-1,1}(\pi) + a_{\pi(j)}= l_{j-1,2}(\pi) + a_{\pi(j)} - R_{j-1}(\pi) \nonumber\\
    && \leq l_{j-1,2}(\pi) + a_{\pi(j)} -  \frac{1}{2}a_{max} \leq  l_{j-1,2}(\pi) = S_{\pi(j)}^2(\sigma).\nonumber
\end{eqnarray}
The last inequality follows from the fact that $a_i\leq \frac{1}{2}a_{max}$ for any $i\in L_0'$. Thus, the operations of the job $\pi(j)$ do not overlap.

Let $\pi(h_2) < \pi(h_0)$ be the last job from $L_2'$ that precedes the job $\pi(h_0)$ in the permutation $\pi.$
 Thus, the jobs $\pi(h_2+1), \ldots,  \pi(h_0)$ belong to  $L_0'.$ We have
\begin{eqnarray}
&&R_{h_0}(\pi) = R_{h_2-1}(\pi) + \Delta_{\pi(h_2)} +\sum_{i=h_2+1}^{h_0} \Delta_{\pi(i)}\nonumber\\
&&= R_{h_2-1}(\pi) + \Delta_{\pi(h_2)} + \mu(h'),\nonumber
\end{eqnarray}
where $h'=|L_0'|+1-h_0+h_2$ is the position of the job $\pi(h_2+1)$ in the permutation $\pi_0.$

Taking into account (\ref{eqnNB}) we obtain
\begin{equation}
R_{h_0}(\pi)  \geq \min\{2a_{\max} + \Delta_{\pi(h_2)} + \mu(h'),  a_{max} \} \geq a_{max}.\label{eqn:lb2}
\end{equation}

\textbf{Case 3(b): $h_0 + 1 \leq j \leq h_1.$ } Consider two sub-cases, which are defined by the set the job $\pi(j)$ belongs to: either $\pi(j)\in L_2'$ or $\pi(j)\in L_1'$.

\textbf{Case 3.1(b): $\pi(j) \in L_2'$.}
Due to step $11$ of the algorithm, we have $R_{j-1}(\pi)\geq 2a_{max}$.  Thus we get
\begin{eqnarray}
  C_{\pi(j)}^1(\sigma) && = S_{\pi(j)}^1(\sigma) + a_{\pi(j)} = l_{j-1,1}(\pi) + a_{\pi(j)}\nonumber\\
  &&= l_{j-1,2}(\pi) + a_{\pi(j)} - R_{j-1}(\pi)\nonumber\\
  && \leq l_{j-1,2}(\pi) + a_{\pi(j)} -  2a_{max} \leq S_{\pi(j)}^2(\sigma). \nonumber
\end{eqnarray}
Thus, the operations of the job $\pi(j)$ do not overlap. Moreover,
 \[
 R_{j}(\pi) = R_{j-1}(\pi) +b_j - a_j \geq 2a_{max} - a_{max} \geq  a_{max}.\label{eqn:lb3}
\]
Taking into account (\ref{eqn:R_up}) and (\ref{eqn:lb2}), we have $R_{j}(\pi) \geq  a_{max}$ for all $h_0  \leq j \leq h_1.$

\textbf{Case 3.2(b): $\pi(j) \in L_1'$.} We have
\begin{eqnarray}
  C_{\pi(j)}^1(\sigma) &&= S_{\pi(j)}^1(\sigma) + a_{\pi(j)} = l_{j-1,1}(\pi) + a_{\pi(j)}\nonumber\\
  &&= l_{j-1,2}(\pi) + a_{\pi(j)} - R_{j-1}(\pi)  \nonumber\\
  && \leq l_{j-1,2}(\pi) + a_{\pi(j)} -  a_{max} \leq S_{\pi(j)}^2(\sigma).\nonumber
\end{eqnarray}
Thus, the operations of the job $\pi(j)$ do not overlap.

\textbf{Case 4(b): $h_1+1 \leq j \leq n' $. } Finally, we observe that the $h_1$ first jobs are the same for $\pi$ and $\pi^J.$
Moreover, the machines $M_1$ and $M_2$ work without idle time in both schedules $\sigma$ and $\sigma'.$
Thus, $C_{\pi(h_1)}^1(\sigma)=C_{\pi^J(h_1)}^1(\sigma')$ and $C_{\pi(h_1)}^2(\sigma)=C_{\pi^J(h_1)}^2(\sigma').$
Moreover, $\pi(j) = \pi^J(j)$ for all $j > h_1.$ Hence, the schedules $\sigma$ and $\sigma'$  are the same for the last $n'-h_1$
jobs and feasibility of $\sigma'$  implies that operations of job $\pi(j)$ do not overlap for $j=h_1+1,\ldots, n'.$\\

\noindent\textbf{(c)}:  Firstly, we  obtain an upper bound for the value of $R_{j}(\pi)$ for all $j,$ $0 \leq j \leq n',$
then we use this upper bound to show that in the schedule $\sigma$ the buffer constraints are not violated. Observe that $R_{0}(\pi)=0$. Consider three sub-cases, which are defined by the set the job $\pi(j)$ belongs to: either $\pi(j)\in L_0'$ or $\pi(j)\in L_1'$ or $\pi(j)\in L_2'$.

\textbf{Case 1(c):  $\pi(j) \in L_0'$. } According to the algorithm either $R_{j-1} (\pi) < \frac{3}{2}a_{max}$ or
$ \frac{3}{2}a_{max} \leq R_{j-1}(\pi)  <  2a_{max}$ and $R_{j-1}(\pi)  + \Delta_{\pi_2(i_2)} + \mu(i_0) < a_{max},$
where $i_0$ and $i_2$ are specified according to Algorithm 1.
If $R_{j-1}(\pi)  < \frac{3}{2}a_{max}$ , then
\begin{equation}\label{UB1}
R_{j}(\pi) = R_{j-1}(\pi) + b_j - a_j \leq \frac{3}{2}a_{max} + b_{max}. 						
\end{equation}
If $ \frac{3}{2}a_{max} \leq R_{j-1}(\pi)  <  2a_{max},$ we get
\begin{eqnarray}
R_{j}(\pi) &&= R_{j-1}(\pi) + b_{\pi(j)} - a_{\pi(j)}\nonumber \leq R_{j-1}(\pi) +  \mu(i_0)\nonumber\\
&&\leq a_{max} - \Delta_{\pi_2(i_2)} = a_{max} - b_{\pi_2(i_2)} + a_{\pi_2(i_2)} \leq 2a_{max}. \label{UB2}						
\end{eqnarray}

\textbf{Case 2(c):  $\pi(j) \in L_1'$. }  According to the algorithm $R_{j-1}(\pi)  <  2a_{max}.$
So, we obtain
\begin{equation}\label{UB3}
R_{j}(\pi) = R_{j-1}(\pi) + b_j - a_j \leq 2a_{max} + b_{max} - \frac{1}{2}a_{max} =  \frac{3}{2}a_{max} + b_{max}.						
\end{equation}

\textbf{Case 3(c):  $\pi(j) \in L_2'$. } From (\ref{UB1}) - (\ref{UB3}) we have
\begin{equation} \label{UB}
R_{i}(\pi) \leq  \frac{3}{2}a_{max} + \max\{\frac{1}{2}a_{max}, b_{max}\}.
\end{equation}
   for all $i \in L_0' \cup  L_1'.$ Since $R_{i-1}(\pi)  \geq R_{i}(\pi)$ for any $j$ such that $\pi(j) \in L_2'$
   the inequality (\ref{UB}) holds for all $j.$

For any $j\in N'$ consider the buffer consumption at its starting time $S_j^1(\sigma)$. Let $k(j)$ be
the job with the smallest completion time such that $C_{k(j)}^2(\sigma)\geq S_j^1(\sigma)$. The buffer consumption at $S_j^1(\sigma)$ does
not exceed the buffer requirements  of job $k(j)$ and  all jobs within interval $[C_{k(j)}^1(\sigma),C_{k(j)}^2(\sigma)]$ and the job $j$.
By virtue of (\ref{Buffer/5}) and  (\ref{UB}) this buffer load does not exceed
\[
    a_{max}+R_{k(j)}(\pi)+a_{max} \leq 3.5a_{max} +\max\{0.5a_{max}, b_{max}\} =\Omega,
\]
hence the buffer capacity is observed every time a job starts its processing.

Since, both machines proceed jobs without idle time, the makespan of $\sigma$ coincides with the load of the machine
and thus the schedule $\sigma$ is an optimal schedule.
\qed

\end{proof}

It remains to note, that after removing from the schedule $\sigma$ all jobs from set $X\cup Y$ we obtain a feasible schedule for the original instance $I$ of the problem.
Further, it is easy to check that if permutations $\pi_0,$ $\pi_1,$ and $\pi_2,$ are specified, the running time of Algorithm 1 linearly depends on $n'$.  In particular, it has been shown in \cite{kononov2019flow} that  if $a_{max}=b_{max}$, then $n' \leq 2n+1$, and such instance of the problem can be solved in $O(n\log n)$ time. However, if the ratio between $a_{max}$ and $b_{max}$ is large enough, the number of jobs in $X \cup Y$, and hence the running time of Algorithm 1 may increase significantly. The following implementation of the algorithm will allow to avoid the increase of the running time.

Firstly, we show that instead of inserting one job from $X$ at each iteration of the first while-loop, we will schedule several jobs from $X$ at the same time if they are to be scheduled sequentially.
 Consider an iteration $i$ of the while-loop on which $i_0 < |L_0'|$ and $\pi(i-1) \in L_2'.$ Denote by $b=\frac{Idle_2}{b_{max}}$, and define
$$\kappa_1 = \min \{k| R_{i-1}(\pi) + kb \geq 2 a_{max}\}$$ and
$$\kappa_2 = \min \{k| R_{i-1}(\pi) + kb \geq 1.5 a_{max} \& R_{i-1}(\pi)  + \Delta_{\pi_2(i_2)} + \mu(i_0)\geq a_{max} \},$$
where $i_0$ is specified according to Algorithm 1.
Observe that the calculation of $\kappa_1$ and  $\kappa_2$ requires $O(1)$ time. Let $\kappa=\min\{ \kappa_1,  \kappa_2, |L_0'| - i_0\}.$ Thus, instead of inserting $\kappa$ jobs from the set $L_0'$ we assign to the schedule one job with the processing time $0$ on the first machine and the processing time $\kappa b$ on the second machine. The number of such new auxiliary jobs does not exceed $|L_2|+1$ and therefore does not exceed $n.$ Observe that after deleting these auxiliary jobs, the resulting schedule will coincide with the schedule $\sigma$ after all jobs  from $X$ were deleted, and therefore the resulting schedule is optimal.

The reduction of the time required for scheduling jobs in $Y$ can be achieved similar to the aggregation of jobs in $X$. If  some job $j\in Y$ is to be scheduled in the $i$th position in $\pi$, then the number of jobs from $Y$ that should be scheduled between $j$ and the next job from $N'\setminus Y$ (or the number of all unscheduled jobs in $Y$ when all jobs in $N'\setminus Y$ have been already scheduled) is computed, and the parameters $l_{i,1}(\pi)$ and $l_{i,2}(\pi)$ are updated accordingly without assigning the starting times for the corresponding jobs from $Y$.

So, taking into account Lemma~\ref{lem:alg1buidlds opt} we get the following result.

\begin{theorem}\label{thm:polyn_case}
Any instance of PP-problem that satisfies (\ref{Buffer/5}) can be solved  in $O(n \log n)$ time.
\end{theorem}

\begin{cor}
Any instance of PP-problem that satisfies $\Omega \geq 4.5\max\{a_{max}, b_{max}\}$ can be solved  in $O(n \log n)$ time.
\end{cor}

\section{Conclusion}\label{sec:conclusion}

This paper contributes to the efforts aimed at establishing the borderline between polynomially solvable and NP-hard cases of the two-machine flow shop problem with the makespan objective function and job-dependent storage requirements by analysing how the variation of a lower bound on the storage capacity, computed as a function of processing times, affects the computational complexity. The paper also strengthens the result in \cite{kononov2019flow} by significantly reducing the lower bound that guarantees the polynomial-time solvability. One of the directions of the further research could be the analysis of the worst-case behaviour of the polynomial-time algorithm, presented in this paper, when this algorithm is applied to the general case of the makespan minimisation problem. Another direction of the future research could be the development of efficient optimisation algorithms for the general case of the makespan minimisation problem.
\begin{acknowledgements}
The research of the first author was supported by the program of fundamental scientific researches of the SB RAS, project 0314-2019-0014.
\end{acknowledgements}
\bibliographystyle{spmpsci}
\bibliography{bibliography-theisis-3}

\begin{thebibliography}{10}
\providecommand{\url}[1]{{#1}}
\providecommand{\urlprefix}{URL }
\expandafter\ifx\csname urlstyle\endcsname\relax
  \providecommand{\doi}[1]{DOI~\discretionary{}{}{}#1}\else
  \providecommand{\doi}{DOI~\discretionary{}{}{}\begingroup
  \urlstyle{rm}\Url}\fi

\bibitem{berlinska2015scheduling}
Berli{\'n}ska, J.: Scheduling for data gathering networks with data
  compression.
\newblock European Journal of Operational Research \textbf{246}(3), 744--749
  (2015)

\bibitem{berlinska2020heuristics}
Berli{\'n}ska, J.: Heuristics for scheduling data gathering with limited base
  station memory.
\newblock Annals of Operations Research \textbf{285}(1), 149--159 (2020)

\bibitem{berlinska2020two}
Berli{\'n}ska, J., Kononov, A., Zinder, Y.: Two-machine flow shop with dynamic
  storage space.
\newblock Optimisation Letters pp. 1--22 (2020)

\bibitem{blazewicz1983scheduling}
Blazewicz, J., Lenstra, J.K., Kan, A.R.: Scheduling subject to resource
  constraints: classification and complexity.
\newblock Discrete Applied Mathematics \textbf{5}(1), 11--24 (1983)

\bibitem{brucker2003flow}
Brucker, P., Heitmann, S., Hurink, J.: Flow-shop problems with intermediate
  buffers.
\newblock OR Spectrum \textbf{25}(4), 549--574 (2003)

\bibitem{brucker2012resource}
Brucker, P., Knust, S.: Complex Scheduling.
\newblock Springer (2012)

\bibitem{emmons2013hybrid}
Emmons, H., Vairaktarakis, G.: Flow Shop Scheduling.
\newblock Springer (2013)

\bibitem{ernst2018flexible}
Ernst, A., Fung, J., Singh, G., Zinder, Y.: Flexible flow shop with dedicated
  buffers.
\newblock Discrete Applied Mathematics  (2018)

\bibitem{fung2015capacity}
Fung, J., Singh, G., Zinder, Y.: Capacity planning in supply chains of mineral
  resources.
\newblock Information Sciences \textbf{316}, 397--418 (2015)

\bibitem{fung2016permutation}
Fung, J., Zinder, Y.: Permutation schedules for a two-machine flow shop with
  storage.
\newblock Operations Research Letters \textbf{44}(2), 153--157 (2016)

\bibitem{GU2018143}
Gu, H., Kononov, A., Memar, J., Zinder, Y.: Efficient lagrangian heuristics for
  the two-stage flow shop with job dependent buffer requirements.
\newblock Journal of Discrete Algorithms \textbf{52-53}, 143 -- 155 (2018)

\bibitem{gu2019improved}
Gu, H., Memar, J., Zinder, Y.: Improved lagrangian relaxation-based
  optimisation procedure for scheduling with storage.
\newblock IFAC-PapersOnLine \textbf{52}(13), 100--105 (2019)

\bibitem{johnson1954optimal}
Johnson, S.M.: Optimal two-and three-stage production schedules with setup
  times included.
\newblock Naval research logistics quarterly \textbf{1}(1), 61--68 (1954)

\bibitem{kononov2012quantity}
Kononov, A., Hong, J.S., Kononova, P., Lin, F.C.: Quantity-based
  buffer-constrained two-machine flowshop problem: active and passive prefetch
  models for multimedia applications.
\newblock Journal of Scheduling \textbf{15}(4), 487--497 (2012)

\bibitem{kononov2019flow}
Kononov, A., Memar, J., Zinder, Y.: Flow shop with job--dependent buffer
  requirements—a polynomial--time algorithm and efficient heuristics.
\newblock In: International Conference on Mathematical Optimisation Theory and
  Operations Research, pp. 342--357. Springer (2019)

\bibitem{kononova2013variable}
Kononova, P., Kochetov, Y.A.: The variable neighborhood search for the two
  machine flow shop problem with a passive prefetch.
\newblock Journal of Applied and Industrial Mathematics \textbf{7}(1), 54--67
  (2013)

\bibitem{le2019iterated}
Le, H.T., Geser, P., Middendorf, M.: An iterated local search algorithm for the
  two-machine flow shop problem with buffers and constant processing times on
  one machine.
\newblock In: European Conference on Evolutionary Computation in Combinatorial
  Optimisation (Part of EvoStar), pp. 50--65. Springer (2019)

\bibitem{le2020iterated}
Le, H.T., Geser, P., Middendorf, M.: Iterated local search and other algorithms
  for buffered two-machine permutation flow shops with constant processing
  times on one machine.
\newblock Evolutionary Computation pp. 1--25 (2020)

\bibitem{lin2009two}
Lin, F.C., Hong, J.S., Lin, B.M.: A two-machine flowshop problem with
  processing time-dependent buffer constraints—an application in multimedia
  presentations.
\newblock Computers \& Operations Research \textbf{36}(4), 1158--1175 (2009)

\bibitem{lin2013sequence}
Lin, F.C., Hong, J.S., Lin, B.M.: Sequence optimisation for media objects with
  due date constraints in multimedia presentations from digital libraries.
\newblock Information Systems \textbf{38}(1), 82--96 (2013)

\bibitem{lin2008minimize}
Lin, F.C., Lai, C.Y., Hong, J.S.: Minimize presentation lag by sequencing media
  objects for auto-assembled presentations from digital libraries.
\newblock Data \& Knowledge Engineering \textbf{66}(3), 382--401 (2008)

\bibitem{min2019two}
Min, Y., Choi, B.C., Park, M.J.: Two-machine flow shops with an optimal
  permutation schedule under a storage constraint.
\newblock Journal of Scheduling pp. 1--10 (2019)

\bibitem{papadimitriou1980flowshop}
Papadimitriou, C.H., Kanellakis, P.C.: Flowshop scheduling with limited
  temporary storage.
\newblock Journal of the ACM (JACM) \textbf{27}(3), 533--549 (1980)

\bibitem{pinedo2016scheduling}
Pinedo, M.L.: Scheduling: Theory, Algorithms, and Systems.
\newblock Springer (2016)

\bibitem{witt2007simple}
Witt, A., Vo{\ss}, S.: Simple heuristics for scheduling with limited
  intermediate storage.
\newblock Computers \& Operations Research \textbf{34}(8), 2293--2309 (2007)

\bibitem{zinder20215}
Zinder, Y., Kononov, A., Fung, J.: A 5-parameter complexity classification of
  the two-stage flow shop scheduling problem with job dependent storage
  requirements.
\newblock Journal of Combinatorial Optimization pp. 1--34 (2021)

\end{thebibliography}

\end{document}